\documentclass[12pt, reqno]{amsart}
\usepackage{amsmath, amsthm, amscd, amsfonts, amssymb, graphicx, color}
\usepackage[bookmarksnumbered, colorlinks, plainpages]{hyperref}

\newtheorem{theorem}{Theorem}[section]
\newtheorem{lemma}[theorem]{Lemma}
\newtheorem{proposition}[theorem]{Proposition}
\newtheorem{corollary}[theorem]{Corollary}
\theoremstyle{definition}
\newtheorem{definition}[theorem]{Definition}
\newtheorem{example}[theorem]{Example}

\theoremstyle{remark}
\newtheorem{remark}[theorem]{Remark}
\numberwithin{equation}{section}

\newcommand{\thmref}[1]{Theorem~\ref{#1}}

\newcommand{\lemref}[1]{Lemma~\ref{#1}}
\newcommand{\propref}[1]{Proposition~\ref{#1}}
\newcommand{\corref}[1]{Corollary~\ref{#1}}

 \def\Tr{\mathop{\rm Trace}\nolimits}

 \def\R{{\bf{R}}}

\def\defequal{\stackrel{\mathrm{def}}{=}}

\title{Continuity of the dual Haar measure}

\author{Jean N. Renault}
\address{Institut Denis Poisson (UMR 7013)\\
  Universit\'e d'Orl\'eans et CNRS \\ 45067
  Orl\'eans Cedex 2, FRANCE}
\email{jean.renault@univ-orleans.fr}

\usepackage{pdfsync}
\begin{document}
\vskip5mm
\begin{abstract} Given a locally compact group bundle, we show that the system of  the Plancherel weights of their C*-algebras is lower semi-continuous. As a corollary, we obtain that the dual Haar sytem of a continuous Haar system of a locally compact abelian group bundle is also continuous.
\end{abstract}\maketitle
\markboth{Jean Renault}
{Dual Haar system}
\renewcommand{\sectionmark}[1]{}
Let $G$ be a locally compact abelian group with Haar measure $\lambda$. The dual Haar measure $\hat\lambda$ on the dual group $\hat G$ is the Haar measure on $\hat G$ which makes the Fourier transform defined by
\[{\mathcal F}(f)(\chi)=\int f(\gamma)\overline{\chi(\gamma)}d\lambda(\gamma)\]
for $f$ in the space $C_c(G)$ of complex-valued continuous functions with compact support on $G$, an isometry from $L^2(G,\lambda)$ to $L^2(\hat G,\hat\lambda)$. Suppose now that $p:G\to X$ is a locally compact abelian group bundle. Here, we mean that $G$ and $X$ are locally compact Hausdorff spaces, $p$ is a continuous surjection and the fibres $G_x=p^{-1}(x)$ are abelian groups. We also require that the multiplication, as a map from $G\times_X G$ to $G$, and the inverse, as a map from $G$ to $G$ are continuous. A Haar system for $G$ is a family of measures $(\lambda^x)_{x\in X}$, where for all $x\in X$, $\lambda^x$ is a Haar measure of $G_x$; it is said to be continuous if for all $f\in C_c(G)$, the function $x\mapsto\int f d\lambda^x$ is continuous. Then we can form the dual group bundle $\hat p:\hat G\to X$. As a set, $\hat G$ is the disjoint union of the dual groups $\hat G_x$. Moreover, it can be identified with the spectrum of the commutative C*-algebra $C^*(G)$, where $G$ is viewed as a locally compact groupoid with Haar system $(\lambda^x)_{x\in X}$. Hence it is endowed with a locally compact topology. It can be checked (\cite[Corollary 3.4]{mrw:continuous-trace III}) that $\hat p:\hat G\to X$ is indeed a locally compact abelian group bundle in the above sense. Of course, one expects that the dual Haar system $(\hat\lambda^x)_{x\in X}$ is continuous. This is stated as Proposition 3.6 of \cite{mrw:continuous-trace III}. However it was recently pointed to the authors by Henrik Kreidler that their proof is defective. This note corrects this and gives a more general result, based on the fact that the dual Haar measure $\hat\lambda$ on the dual group $\hat G$ of a locally compact abelian group $G$, viewed as its Haar weight and defined as the canonical weight of $L^\infty(\hat G)$ acting on $L^2(\hat G,\hat\lambda)$, corresponds under the Fourier transform to the Plancherel weight of $G$, defined as the canonical weight of the von Neumann algebra $VN(G)$ of $G$ acting on $L^2(G,\lambda)$. Therefore one can consider a locally compact group bundle $p:G\to X$ where the fibres $G_x$ are no longer abelian. Our main result is \propref{main} which says that the Plancherel weight of $G_x$ varies continuously in a suitable sense. This lead us to the definition of a lower semi-continuous $C_0(X)$-weight on a $C_0(X)$-C*-algebra which we illustrate by three examples.

The first section recalls the construction of the canonical weight of a left Hilbert algebra and the properties which are needed to prove the crucial \corref{key}. In the second section, we consider the case of a locally compact group bundle and prove our main result. This example motivates the definition of a $C_0(X)$-weight of a $C_0(X)$-C*-algebra, given in the third section. The fact that the Plancherel $C_0(X)$-weight of the C*-algebra of a group bundle is densely defined and lower semi-continuous gives   Proposition 3.6 of \cite{mrw:continuous-trace III}.

\section{The Plancherel weight of a locally compact group}

We recall first some elements of Tomita-Takesaki's theory, using the standard notation from \cite{tak:tom}. Given a left Hilbert algebra $\mathcal A$ where the product, the involution and the scalar product are respectively denoted by $ab$, $a^\sharp$ and $(a\,|\,b)$, we denote by $\mathcal H$ the Hilbert space completion of $\mathcal A$ and by $\pi:{\mathcal A}\to{\mathcal L}({\mathcal H})$ the left representation. We denote by ${\mathcal M}=\pi({\mathcal A})''$ the left von Neumann algebra of $\mathcal A$. We denote by $S$ the closure of the involution $a\mapsto a^\sharp$ and by $F$ its adjoint. The domain of $S$ [resp. $F$] is denoted by ${\mathcal D}^\sharp$ [resp. ${\mathcal D}^\flat$] and one writes $\xi^\sharp=S\xi$ [resp. $\eta^\flat=F\eta$] for $\xi\in {\mathcal D}^\sharp$  [resp. for $\eta\in {\mathcal D}^\flat$]. An element $\eta\in {\mathcal H}$ is called right bounded if there exists a bounded operator $\pi'(\eta)$ on $\mathcal H$ such that $\pi'(\eta)a=\pi(a)\eta$ for all $a\in{\mathcal A}$. One then writes $\xi\eta=\pi'(\eta)\xi$ for all $\xi\in{\mathcal H}$. The set of right bounded elements is denoted by ${\mathcal B}'$. One shows that ${\mathcal A}'={\mathcal B}'\cap {\mathcal D}^\flat$ with involution $\xi^\flat$ is a right Hilbert algebra in the same Hilbert space $\mathcal H$. An element $\xi\in {\mathcal H}$ is called left bounded if there exists a bounded operator $\pi(\xi)$ on $\mathcal H$ such that $\pi(\xi)\eta=\pi'(\eta)\xi$ for all $\eta\in{\mathcal A}'$. One then writes $\xi\eta=\pi(\xi)\eta$ for all $\eta\in{\mathcal H}$. If $\xi$ is left bounded and $\eta$ is right bounded, the notation is consistent: $\pi(\xi)\eta=\xi\eta=\pi'(\eta)\xi$. The set of left bounded elements is denoted by ${\mathcal B}$. We shall need the following well-known lemmas.

\begin{lemma}\label{left bounded} Let ${\mathcal B}$ be the set of left bounded elements of $\mathcal H$. Then,
\begin{enumerate}
 \item ${\mathcal B}$ is a linear subspace of $\mathcal H$ containing $\mathcal A$.
 \item $\pi({\mathcal B})$ is contained in the left von Neumann algebra ${\mathcal M}$.
 \item ${\mathcal B}$ is stable under ${\mathcal M}$. More precisely, for $T\in {\mathcal M}$ and $\xi\in {\mathcal B}$, $\pi(T\xi)=T\pi(\xi)$.
\end{enumerate}
 \end{lemma}
 
 \begin{proof}
See \cite[Section 2]{com:poids associe}. 
\end{proof}

\begin{lemma}\label{adjoint} Let $\xi,\xi'\in{\mathcal B}$ such that $\pi(\xi')=\pi(\xi)^*$. Then $\xi$ belongs to ${\mathcal D}^\sharp$ and $\xi'=\xi^\sharp$.
\end{lemma}

\begin{proof} For all $\eta_1,\eta_2\in {\mathcal A}'$, one has:
 \[\begin{array}{lll}
(\xi\,|\, \eta_1\eta_2^\flat)&=(\xi\,|\,\pi'(\eta_2^\flat)\eta_1)\\
&=(\xi\,|\,\pi'(\eta_2)^*\eta_1)\\
&=(\pi'(\eta_2)\xi\,|\,\eta_1)\\
&=(\pi(\xi)\eta_2\,|\,\eta_1)\\
&=(\pi(\xi')^*\eta_2\,|\,\eta_1)\\
&=(\eta_2\,|\,\pi(\xi')\eta_1)\\
&=(\eta_2\,|\,\pi'(\eta_1)\xi')\\
&=(\pi'(\eta_1)^*\eta_2\,|\,\xi')\\
&=(\pi'(\eta_1^\flat)\eta_2\,|\,\xi')\\
&=(\eta_2\eta_1^\flat\,|\,\xi')\\
\end{array}\]
According to \cite[Lemma 3.4]{tak:tom} and the comments which follow, this suffices to conclude.
\end{proof}

\begin{definition} The canonical weight $\tau$ of the left Hilbert algebra $\mathcal A$ is the map $\tau: {\mathcal M}_+\to [0,\infty]$ defined by
\[\tau(T)=\left\{\begin{array}{ccc}
\|\xi\|^2\,&{\rm if}&\,\exists\,\xi\in {\mathcal B}: T=\pi(\xi)^*\pi(\xi)\\
\infty &{\rm otherwise}&
\end{array}\right.\]

\end{definition}

\begin{lemma}\cite[Th\'eor\`eme 2.11]{com:poids associe}
 The above canonical weight $\tau$ is well-defined and is a faithful, semi-finite, $\sigma$-weakly lower semi-continuous weight on ${\mathcal M}_+$.
\end{lemma}

The following lemma is certainly well-known but I did not find a reference for it.

\begin{lemma}\label{key lemma} For all $\xi\in {\mathcal B}$, one has
\[\tau(\pi(\xi)\pi(\xi)^*)=\left\{\begin{array}{ccc}
\|\xi^\sharp\|^2\,&{\rm if}&\,\xi\in {\mathcal D}^\sharp\\
\infty &{\rm otherwise}&
\end{array}\right.\]
\end{lemma}

\begin{proof} Suppose that $\xi\in {\mathcal D}^\sharp$. Then $\xi$ belongs to the full Hilbert algebra ${\mathcal A}''={\mathcal B}\cap {\mathcal D}^\sharp$ of $\mathcal A$. Therefore, $\xi^\sharp$ is also left bounded and $\pi(\xi)^*=\pi(\xi^\sharp)$. Thus,
 \[\tau(\pi(\xi)\pi(\xi)^*)=\tau(\pi(\xi^\sharp)^*\pi(\xi^\sharp))=\|\xi^\sharp\|^2\]
Suppose now that the left handside is finite. Then there exists $\eta\in {\mathcal B}$ such that $\pi(\xi)\pi(\xi)^*=\pi(\eta)^*\pi(\eta)$.
 Let $\pi(\xi)=U |\pi(\xi)|$ [resp. $\pi(\eta)=V |\pi(\eta)|$]  be the polar decomposition of $\pi(\xi)$ [resp.  $\pi(\eta)$]. All these operators belong to the left von Neumann algebra $\mathcal M$. Our assumption implies that $|\pi(\eta)|=U|\pi(\xi)|U^*$. Therefore,
 \[\begin{array}{lll}
(U^*V^*\pi(\eta))^*&=\pi(\eta)^*VU\\
&=|\pi(\eta)|V^*VU\\
&=|\pi(\eta)|U\\
&=U|\pi(\xi)|\\
&=\pi(\xi))\\
\end{array}\]
Moreover, according to  \lemref{left bounded}, $U^*V^*\pi(\eta)=\pi(U^*V^*\eta)$ and $\zeta=U^*V^*\eta$ is left bounded. Since $\pi(\xi)=\pi(\zeta)^*$ we deduce from \lemref{adjoint} that $\xi\in{\mathcal D}^\sharp$.
\end{proof}

We consider now the left Hilbert algebra associated with the left regular representation of a locally compact group $G$ endowed with a left Haar measure $\lambda$. Using the framework of  \cite [pages 56-58]{ren:groupoid}, we choose the left Hilbert algebra ${\mathcal A}=C_c(G)$, where the product is the usual convolution product, the involution, denoted by $f^*$ instead of $f^\sharp$, is $f^*(\gamma)=\overline{f(\gamma^{-1})}$ and the scalar product is
\[(f\,|\,g)=\int f  \overline g d\lambda^{-1}.\]
The left representation is denoted by $L$ instead of $\pi$. It acts on the Hilbert space ${\mathcal H}=L^2(G,\lambda^{-1})$ by $L(f)g=f*g$ for $f,g\in C_c(G)$. The left von Neumann algebra ${\mathcal M}$ is the group von Neumann algebra $VN(G)$. The canonical weight $\tau$ of this left Hilbert algebra is also called the Plancherel weight of the group $G$ (\cite{str:modular,tak:II}). We also refer the reader to \cite[Section 2]{haa:dual} for a deep study of its properties. It satisfies
\[\tau(L(f^**f))=\|f\|^2=\int |f|^2 d\lambda^{-1}=(f^**f)(e)\]
for $f\in C_c(G)$. When $G$ is abelian, we can use the Fourier transform to compute the Plancherel weight (see \cite[VII.3]{tak:II}). The canonical weight $\hat \tau$ of the commutative Hilbert algebra $L^2(\hat G,\hat\lambda)\cap L^\infty(\hat G)$ is the integral with respect to the dual Haar measure $\hat \lambda$ on $L^\infty(\hat G)_+$. This weight is called the Haar weight of the dual group $\hat G$. The Fourier transform $\mathcal F$ implements an isomorphism between the full left Hilbert algebra of the regular representation of $G$ and the Hilbert algebra $L^2(\hat G,\hat\lambda)\cap L^\infty(\hat G)$. Therefore, $\tau(T)=\hat\tau({\mathcal F}\circ T\circ{\mathcal F}^{-1})$ for $T\in VN(G)$. One has in particular
 \[\tau(L(a))=\int {\mathcal G}(a)d\hat\lambda,\]
 for $a\in C^*(G)_+$ where ${\mathcal G}: C^*(G)\to C_0(\hat G)$ is the Gelfand transform.\\

Let us express \lemref{key lemma} in the case of the left regular representation of a locally compact group $G$ with a left Haar measure $\lambda$.

\begin{corollary}\label{key} For all left bounded element $\xi$ of $L^2(G,\lambda^{-1})$, one has 
\[\tau(L(\xi)L(\xi)^*)=\int |\xi|^2 d\lambda.\]
\end{corollary} 

\begin{remark}
One can give a more direct proof of this result  since the equality $L(\xi)^*\eta=\xi^**\eta$ can be established by usual integration techniques and \cite[Lemme 3.1]{eym:Fourier}. However, its natural framework is Tomita-Takesaki's theory.
\end{remark}
 
\section{The Plancherel weight of a group bundle}

We consider now the case of a locally compact group bundle $p:G\to X$. We assume that the groups $G_x$ have a left Haar measure $\lambda^x$ such that $x\mapsto \lambda^x$ is continuous. This is the particular case of a locally compact groupoid with Haar system where the range and source maps coincide. Thus, we can construct its C*-algebra $C^*(G)$ as usual. If the groups $G_x$ are abelian, it is a commutative C*-algebra. Then the Gelfand transform identifies it with $C_0(\hat G)$. The space $\hat G$ is the total space of the bundle $\hat p:\hat G\to X$, where the fibre $\hat G_x$ is the dual group of $G_x$. For $a\in C^*(G)_+$ and $x\in X$, we define
\[{\mathcal T}(a)(x)\defequal\int {\mathcal G}(a)d\widehat{\lambda^x}=\int {\mathcal G}_x(a_x)d\widehat{\lambda^x}=\tau_x\circ L_x(a_x),\]
where $L_x$ is the left regular representation of $G_x$ and $\tau_x$ is the Plancherel weight of $G_x$.
The last expression is defined when the groups are non-abelian and we turn now to this case.

\begin{lemma} Let $G\to X$ be a locally compact group bundle with Haar system $\lambda=(\lambda^x)_{x\in X}$. Then,
\begin{enumerate}
\item $x\mapsto C^*(G_x)$ is an upper semi-continuous field of C*-algebras;
\item its sectional algebra is $C^*(G)$.
 \end{enumerate}
 \end{lemma}
\begin{proof}
 See for example \cite[Section 5]{lr:quantization}.
\end{proof}

Equivalently, this lemma says that $C^*(G)$ is a $C_0(X)$-C*-algebra. We shall view an element $a$ of $C^*(G)$ as a continuous field $x\mapsto a_x$, where $a_x\in C^*(G_x)$. For $a\in C^*(G)_+$, we define as in the commutative case the function 
\[{\mathcal T}(a): x\mapsto \tau_x(L_x(a_x)),\] 
where $L_x$ is the left regular representation of $C^*(G_x)$ and $\tau_x$ is the Plancherel weight of $G_x$. Our main result will be that the function $x\mapsto {\mathcal T}(a)(x)$ is lower semi-continuous. We consider now weights on C*-algebras rather than von Neumann algebras and refer the reader to  \cite{com:poids} for the main definitions and results.
\begin{lemma}
 Let $\varphi:A_+\to [0,\infty]$ be a weight on a C*-algebra $A$. Assume that $(e_i)_{i\in I}$ is an approximate identity for $A$ such that $\|e_i\|\le 1$ for all $i\in I$. Let $a\in A_+$. Then,
\begin{enumerate}
 \item $\varphi(a^{1/2}e_i a^{1/2})\le \varphi(a)$;
 \item  if $\varphi$ is lower semi-continuous, then $\varphi(a)=\sup_i\varphi(a^{1/2}e_i a^{1/2})$.
\end{enumerate}
\end{lemma}

\begin{proof}
For (i), we have $a^{1/2}e_i a^{1/2}\le\|e_i\| a\le a$. Since $\varphi$ is increasing, we obtain the desired inequality.
For (ii), since $a^{1/2}e_i a^{1/2}$ converges to $a$, we have
\[\varphi(a)\le \liminf\varphi(a^{1/2}e_i a^{1/2})\le \sup\varphi(a^{1/2}e_i a^{1/2})\le \varphi(a)\]
\end{proof}

\begin{proposition}\label{main} Let $G\to X$ be a locally compact group bundle with Haar system $\lambda=(\lambda^x)_{x\in X}$. Given $a\in C^*(G)_+$, the function 
\[{\mathcal T}(a):x\mapsto \tau_x(L_x(a_x))\]
 is lower semi-continuous.
 \end{proposition}

\begin{proof} Proposition 2.10 of \cite{mrw:equivalence} gives the existence of an approximate unit $(e_i)$ of $C^*(G)$ where each $e_i$ is a finite sum of elements of the form $f*f^*$ with $f\in C_c(G)$; moreover, according to its construction,  $\|e_i\|_I$ tends to 1, where $\|f\|_I=\max{(\sup_x\int |f| d\lambda^x, \sup_x\int |f| d\lambda_x)}$ and $\lambda_x=(\lambda^x)^{-1}$. Replacing $e_i$ by $(1/\|e_i\|_I)e_i$, we can assume that $\|e_i\|_I=1$ for all $i\in I$. We then have $\|e_i\|\le\|e_i\|_I\le 1$. For all $x\in X$, the image $(e_i(x))$ of $(e_i)$ in $C^*(G_x)$ satisfies the same assumptions. Since the Plancherel weight $\tau_x$ is $\sigma$-weakly lower semi-continuous, we have
\[\tau_x(L_x(a_x))=\sup_i \tau_x(L_x(a_x^{1/2}e_i(x) a_x^{1/2})).\]
We will show that the function $x\mapsto \tau_x(L_x(a_x^{1/2}e_i(x) a_x^{1/2}))$ is lower semi-continuous. This will imply that the function $x\mapsto \tau_x(L_x(a_x))$ is lower semi-continuous as lower upper bound of a family of lower semi-continuous functions. It suffices to show that $x\mapsto \tau_x(L_x(a_x^{1/2}*(f*f^*)_x *a_x^{1/2}))$ is lower semi-continuous for all $f\in C_c(G)$. Let us fix $x\in X$. According to \lemref{left bounded}, $f_x$ and $g_x=L_x(a_x^{1/2})f_x$ are left bounded elements of $L^2(G_x,\lambda_x)$. Thus we have according to \lemref{key lemma} :
\[\tau_x(L_x(a_x^{1/2}*(f*f^*)_x *a_x^{1/2})=\tau_x(L_x(g_x)L_x(g_x)^*)=\int |g_x(\gamma)|^2 d\lambda^x(\gamma).\]
Note that $g:x\mapsto g_x$ is an element of the $C_0(X)$-Hilbert module $L^2(G,\lambda^{-1})$ (see \cite[Section 2]{ks:regular}). If the groups $G_x$ are unimodular, then
\[\int |g(\gamma)|^2 d\lambda^x(\gamma)=\int |g(\gamma)|^2 d\lambda_x(\gamma)\]
depends continuously on $x$. If not, we first observe that for all $\rho\in C_c(G)$, $\rho g$ belongs to $L^2(G,\lambda^{-1})$.
 Let $D: G\to\R_+^*$ be the modular function of the group bundle $G$. It is a continuous homomorphism such that
$\lambda^x=D\lambda_x$ for all $x\in X$ (see \cite[Lemma 2.4]{ikrsw:extensions}). There exists an increasing sequence $(D_n)$ of continuous non-negative functions with compact support which converges pointwise to $D$. Then, for all $x\in X$, $\int D_n(\gamma)|g(\gamma)|^2 d\lambda_x(\gamma)$ is an increasing sequence which converges to
\[\int D(\gamma)|g(\gamma)|^2 d\lambda_x(\gamma)=\int |g(\gamma)|^2d\lambda^x(\gamma)\,.\] 
Therefore, the function $x\mapsto \int |g(\gamma)|^2d\lambda^x(\gamma)$ is lower semi-continuous as a limit of an increasing sequence of continuous functions.
\end{proof}

\section{$C_0(X)$-weights on $C_0(X)$-C*-algebras}

In this section, given a topological space $X$, $LSC(X)_+$ denotes the convex cone of lower semi-continuous functions $f:X\to [0,\infty]$. 

\begin{definition}\label{C_0-weight} Let $X$ be a locally compact Hausdorff space. A $C_0(X)$-weight on a $C_0(X)$-C*-algebra $A$ is a map
\[\Phi: A_+\to LSC(X)_+\]
such that
\begin{enumerate}
 \item $\Phi(a+b)=\Phi(a)+\Phi(b)$ for all $a,b\in A_+$;
 \item $\Phi(ha)=h\Phi(a)$ for all $h\in C_0(X)_+$ and $a\in A_+$.
\end{enumerate}
It is called lower semi-continuous if $a_n\to a$ implies $\Phi(a)\le\liminf\Phi(a_n)$ and densely defined if its domain, defined as
\[P=\{a\in A_+: \Phi(a) \, \hbox{is finite and continuous}\}\]
is dense in $A_+$.  
\end{definition}

In order to include the first example, we need to modify slightly this definition. Given a topological space $X$, we denote by $C_b(X)$ the space of complex-valued bounded continuous functions on $X$. We say that a C*-algebra $A$ is a $C_b(X)$-algebra if it is endowed with a nondegenerate morphism of $C_b(X)$ into the centre of the multiplier algebra of $A$. Then we define a $C_b(X)$-weight by replacing $C_0(X)$ by $C_b(X)$ in the above definition.

\begin{remark} The above definition of a $C_0(X)$-weight on a $C_0(X)$-C*-algebra is motivated by the  examples below; it is different from the definition of a C*-valued weight given in \cite{kus:regular}.  \end{remark}

\begin{lemma}
 The domain $P$ of a $C_b(X)$-weight or of $C_0(X)$-weight is hereditary: if $0\le b\le a$ and $a\in P$, then $b\in P$.
\end{lemma}

\begin{proof} See \cite[Lemme 4.4.2.i]{dix:C*}.
 
\end{proof}

\begin{example} {\it The canonical center-valued trace}\label{center} (cf \cite[Section 3.4]{dix:VN}). Let $A$ be a C*-algebra and let $X={\rm Prim}(A)$ its primitive ideal space endowed with the Jacobson topology. Through the Dauns-Hofmann theorem, we view $A$ as a $C_b(X)$-algebra. Given $a\in A_+$, according to \cite[Proposition 4.4.9]{ped:C*}, one can define $\Phi(a)(P)=\Tr(\pi(a))$ where $P$ is a primitive ideal and $\pi$ is any irreducible representation admitting $P$ as kernel, and $\Phi(a)$ is lower semi-continuous on $X$. Condition (i) of the definition results from the additivity of the trace and condition (ii) from the irreducibility of $\pi$. Thus $\Phi: A_+\to LSC(X)_+$ is a $C_b(X)$-weight. Since the usual $\Tr$ is $\sigma$-weakly lower semi-continuous, $\Phi$ is lower semi-continuous. By definition, $\Phi$ is densely defined if and only if $A$ is a continuous-trace C*-algebra.
 \end{example}
 
\begin{example}\label{group bundle} {\it The Plancherel $C_0(X)$-weight of a group bundle}. This is the example described in the previous section: 

\begin{theorem}\label{group bundle} Let $p:G\to X$ be a locally compact group bundle endowed with a continuous Haar system. Then the map
 \[{\mathcal T}: C^*(G)_+\to LSC(X)_+\]
 such that ${\mathcal T}(a)(x)=\tau_x(L_x(a_x))$ for $a\in C^*(G)_+$ and where $\tau_x$ is the Plancherel weight of $G_x$ and $L_x$ the left regular representation of $G_x$ is a densely defined lower semi-continuous $C_0(X)$-weight, which we call the Plancherel $C_0(X)$-weight of the group bundle.
\end{theorem}

\begin{proof}
 The main point was to show that the range of this map is contained in $LSC(X)_+$, which was done in \propref{main}. The conditions (i) and (ii) of the definition are clear. It is densely defined since its domain contains the elements $f^**f$ where $f\in C_c(G)$, whose linear span is dense in $C^*(G)$. The lower semi-continuity of $\tau_x\circ L_x$ for all $x\in X$ gives the lower semi-continuity of $\mathcal T$.
\end{proof}

\end{example}

\begin{example}\label{commutative} {\it The commutative case}. Let $\pi:Y\to X$ be a continuous, open and surjective map, where $Y$ and $X$ are locally compact Hausdorff spaces. For $x\in X$, let $Y_x=\pi^{-1}(x)$ be the fibre over $x$. Endowed with the fundamental family of continuous sections $C_c(Y)$, $x\mapsto C_0(Y_x)$ is a continuous field of C*-algebras. Its C*-algebra of continuous sections is identified to $C_0(Y)$. Thus $C_0(Y)$ is a $C_0(X)$-C*-algebra. The following proposition gives a description of the densely defined lower semi-continuous $C_0(X)$-weights of $C_0(Y)$.

\begin{proposition}\label{commutative} In the above situation, we have natural one-to-one correspondence between
\begin{enumerate}
 \item the continuous $\pi$-systems of measures, by which we mean families of Radon measures $\alpha=(\alpha_x)_{x\in X}$ on $Y$, where $\alpha_x$ is supported on $\pi^{-1}(x)$, and such that for all $f\in C_c(Y)$, the function $x\mapsto \int f d\alpha_x$ is continuous;
\item the densely defined and lower semi-continuous $C_0(X)$-weights of $C_0(Y)$.
\end{enumerate}
 \end{proposition}

\begin{proof} When $X$ is reduced to a point, this is a well-known result given for example in the introduction of \cite{com:poids}. A crucial point which we shall use again below is that a densely defined weight on $C_0(Y)$ is necessarily finite on $C_c(Y)_+$ because the linear span of its domain of definition is a dense ideal, hence it contains the minimal dense ideal, called the Pedersen ideal, which in our case is $C_c(Y)$ (\cite[5.6.3.]{ped:C*}). The proof of the general case is essentially the same. Suppose that $\alpha=(\alpha_x)_{x\in X}$  is a continuous $\pi$-system. For $f\in C_0(Y)_+$, we define $\Phi(f)(x)=\int f d\alpha_x$. There exists an increasing sequence $(f_n)$ in $C_c(Y)_+$ converging uniformly to $f$. Therefore, $\Phi(f)=\sup_n\Phi(f_n)$ is lower semi-continuous. It is clear that the other properties of a $C_0(X)$-weight are satisfied. Since its domain contains $C_c(Y)$, $\Phi$ is densely defined. The lower semi-continuity of the Radon measures gives the lower semi-continuity of $\Phi$. Conversely, let $\Phi$ be a densely defined and lower semi-continuous $C_0(X)$ weight on $C_0(Y)$. For all $x\in X$, we can define $\varphi_x: C_0(Y_x)_+\to [0,\infty]$ such that $\Phi(f)(x)=\varphi_x(f_{|Y_x})$ for all $x$. Then $\varphi_x$ is a densely defined and lower semi-continuous weight on $C_0(Y_x)$. As we recalled at the beginning, there exists a unique Radon measure $\alpha_x$ such that $\phi_x(f)=\int f d\alpha_x$ for all $f\in C_0(Y_x)$. By assumption, the linear span of the domain of definition $P$ of $\Phi$ is dense. Since it is also an ideal, it contains the Pedersen ideal $C_c(Y)$ of $C_0(Y)$. This shows that $\alpha=(\alpha_x)_{x\in X}$  is a continuous $\pi$-system of measures.
 \end{proof}
 \end{example}

\begin{example}\label{groupoid} {\it The Haar $C_0(X)$-weight of a groupoid}. This example is a particular case of the previous example. Let $G$ be a locally compact groupoid. We use the range map $r:G\to G^{(0)}$ to turn the commutative C*-algebra $C_0(G)$ into a $C_0(X)$-algebra, where $X=G^{(0)}$. Then a continuous Haar system $(\lambda^x)_{x\in X}$ defines a densely defined and lower semi-continuous $C_0(X)$-weight of $C_0(G)$, which in accordance with the group case, can be called the Haar $C_0(X)$-weight of $G$. On the other hand, we cannot define a Plancherel $C_0(X)$-weight for an arbitrary groupoid $G$ since $C^*(G)$ is usually not a $C_0(X)$-algebra.
 
\end{example}
 
 We now have all the elements to prove the Proposition 3.6 in \cite{mrw:continuous-trace III}.

\begin{corollary}\label{continuity}
 Let $p:G\to X$ a locally compact bundle of abelian groups, equipped with a continuous Haar system $\lambda=(\lambda^x)_{x\in X}$. Then the family of dual Haar measures $\hat\lambda=(\hat\lambda^x)_{x\in X}$ is a continuous Haar system for $\hat p:\hat G\to X$.
\end{corollary}

\begin{proof} From \thmref{group bundle}, we know that the Plancherel $C_0(X)$-weight $\mathcal T$ on $C^*(G)$ is lower semi-continuous and densely defined. Therefore $\Phi={\mathcal T}\circ{\mathcal G}^{-1}$, where ${\mathcal G}: C^*(G)\to C_0(\hat G)$ is the Gelfand transform, is a lower semi-continuous and densely defined $C_0(X)$-weight on $C_0(\hat G)$. Moreover, for all $x\in X$, the Plancherel weight $\tau_x$ of $G_x$ correspond to the Haar weight of $\hat G_x$, which is given by the Radon measure  $\hat\lambda^x$. From \propref{commutative}, the $\hat p$-system $\hat\lambda=(\hat\lambda^x)_{x\in X}$ is continuous.
 
\end{proof}

\vskip 5mm
{\it Acknowledgements.} I thank Henrik Kreidler for drawing the erroneous proof of  Proposition 3.6 in \cite{mrw:continuous-trace III} to our attention, Michel Hilsum for discussions which led to the present proof of \corref{continuity} and Dana Williams for sharing another proof and providing inspiring feedback and valuable comments.

\end{document}